\documentclass[reqno]{amsart}

\usepackage{graphicx}
\usepackage{amsmath}
\usepackage{amscd}
\usepackage{amssymb}
\usepackage{amstext}
\usepackage{amsmath}
\usepackage[all]{xy}
\usepackage{yhmath}
\usepackage{mathrsfs}
\usepackage{bbding}
\usepackage{color}

\def\E{\ifmmode{\mathbb E}\else{$\mathbb E$}\fi} %natural numbers
\def\N{\ifmmode{\mathbb N}\else{$\mathbb N$}\fi} %natural numbers%
\def\R{\ifmmode{\mathbb R}\else{$\mathbb R$}\fi} %real numbers
\def\Q{\ifmmode{\mathbb Q}\else{$\mathbb Q$}\fi} %rational numbers
\def\C{\ifmmode{\mathbb C}\else{$\mathbb C$}\fi} %complex numbers
\def\H{\ifmmode{\mathbb H}\else{$\mathbb H$}\fi} %complex numbers
\def\Z{\ifmmode{\mathbb Z}\else{$\mathbb Z$}\fi} %integers
\def\P{\ifmmode{\mathbb P}\else{$\mathbb P$}\fi} %real numbers
\def\T{\ifmmode{\mathbb T}\else{$\mathbb T$}\fi} %real numbers
\def\SS{\ifmmode{\mathbb S}\else{$\mathbb S$}\fi} %real numbers
\def\DD{\ifmmode{\mathbb D}\else{$\mathbb D$}\fi} %real numbers
\def\K{\ifmmode{\mathbb K}\else{$\mathbb K$}\fi}

\def\UU{{\mathcal U}}

\theoremstyle{theorem}
\newtheorem{thm}{Theorem}[section]

\newtheorem{lem}[thm]{Lemma}

\newtheorem{prop}[thm]{Proposition}

\theoremstyle{definition}
\newtheorem{defn}[thm]{Definition}
\newtheorem{rem}[thm]{Remark}

\newtheorem{exm}[thm]{Example}

\numberwithin{equation}{section}
\begin{document}
\title[]{Shrinking good coordinate systems associated to
Kuranishi structures.}
\author{Kenji Fukaya, Yong-Geun Oh, Hiroshi Ohta, Kaoru Ono}

\address{Simons Center for Geometry and Physics,
State University of New York, Stony Brook, NY 11794-3636 U.S.A.
\& Center for Geometry and Physics, Institute for Basic Sciences (IBS), Pohang, Korea} \email{kfukaya@scgp.stonybrook.edu}
\address{Center for Geometry and Physics, Institute for Basic Sciences (IBS), Pohang, Korea \& Department of Mathematics,
POSTECH, Pohang, Korea} \email{yongoh1@postech.ac.kr}
\address{Graduate School of Mathematics,
Nagoya University, Nagoya, Japan} \email{ohta@math.nagoya-u.ac.jp}
\address{Research Institute for Mathematical Sciences, Kyoto University, Kyoto, Japan}
\email{ono@kurims.kyoto-u.ac.jp}

\begin{abstract}
The notion of good coordinate system was introduced by Fukaya and Ono
in \cite{FO} in their construction of virtual fundamental chain via
Kuranishi structure which was also introduced therein. This notion was further
clarified in \cite{fooobook2} in some detail. In those papers no explicit ambient space
was used and hence the process of gluing local Kuranishi charts in the given good
coordinate system was not discussed there. In our more recent writing \cite{foootech, foootech2},
we use an ambient space obtained by gluing the Kuranishi charts. In this note we prove in
detail that we can always shrink the given good coordinate system so that the resulting
`ambient space' becomes Hausdorff.
This note is self-contained and uses only standard facts in general topology.
\end{abstract}

\maketitle
\section{Introduction }
In \cite{FO,fooobook2} the present authors associated a virtual
fundamental chain to a space with Kuranishi structure.
For the construction we used the notion of {\it good coordinate system}.
The process of constructing a good coordinate system out of
Kuranishi structure corresponds to that of choosing and fixing an atlas
consisting of a locally finite covering of coordinate charts in the manifold theory.
\par
In \cite{FO,fooobook2} the process to associate the virtual fundamental
chain to a space with good coordinate system,
is described {\it without} using `ambient space', that is, the
space obtained by gluing Kuranishi charts by coordinate change.
In our more recent writing, \cite{foootech, foootech2},  which
contains further detail of this construction, we describe the same
process using `ambient space', explicitly.
For the description of the construction of virtual fundamental chain using
ambient space, certain properties, especially Hausdorff-ness, of the ambient space
is necessary.
\par
In \cite{FO,fooobook2}, the tools of Kuranishi structure and its associated
good coordinate system are applied to study moduli spaces of stable maps.
The moduli space of stable maps can be very singular in general but
we can embed a small portion thereof at each point of the moduli space
locally into an orbifold which is called a Kuranishi neighborhood.
An element of a Kuranishi neighborhood appearing in such applications is a
`map' with domain a nodal curve satisfying a differential equation, that is,
a slightly perturbed Cauchy-Riemann equation. To write down this perturbed Cauchy-Riemann equation,
one needs to fix various extra data locally in our moduli space. Because of this reason,
the union of Kuranishi neighborhoods cannot be globally regarded as a subset
of certain well-defined  set of maps, and gluing \emph{the given} Kuranishi neighborhoods
to construct an ambient space a priori may not make sense.
The main result of the present article is to show that we can, however, always shrink the given
Kuranishi neighborhoods and the domains of coordinate change and
glue the resulting shrinked neighborhoods to obtain certain reasonable space,
which one may call an `ambient space' or a `virtual neighborhood'.
It also shows that we can always do so, \emph{after some shrinking}, by employing only
elementary general topology arguments, with the originally
given definition of good coordinate system  in \cite{FO,fooobook2}.

\par
Our purpose of writing this short note is to separate the abstract
combinatorial general topology issue from other parts of the story of
Kuranishi structure given in \cite{foootech2} and its implementations, and to clarify the parts of general topology.
This note is self-contained and can be read independently of the previous knowledge of Kuranishi
structures.

\section{Statement}

To make it clear that the arguments of this note do not involve
the properties of orbifolds, vector bundles on them,
the smoothness of the coordinate change and others,
we introduce the following abstract notions that
lie in the realm of general topology and not of manifold theory.

In this note, $X$ is always assumed to be a locally compact separable
metrizable space.
\begin{defn}\label{abKchart}
An {\it abstract K-chart} of $X$ consists of $\mathcal U = (U,S,\psi)$ where
$U$ is a locally compact separable metrizable space, $S \subseteq U$ is a closed subset
and $\psi : S \to X$ is a homeomorphism onto an open subset.
\end{defn}
\begin{defn}\label{abKchange}
Let $\mathcal U_i = (U_i,S_i,\psi_i)$ $(i=1,2)$ be abstract K-charts of $X$.
A {\it coordinate change} from $\mathcal U_1$ to $\mathcal U_2$
consists of $\Phi_{21} = (U_{21},\varphi_{21})$ such that:
\begin{enumerate}
\item $U_{21} \subseteq U_1$ is an open set.
\item $\varphi_{21} : U_{21} \to U_2$ is a topological embedding, i.e., a continuous map
which is a homeomorphism onto its image.
\item $S_1 \cap U_{21} =  \varphi_{21}^{-1}(S_2)$.
Moreover $\psi_2 \circ \varphi_{21} = \psi_1$ on $S_1 \cap U_{21}$
(i.e., whenever both are defined).
\item
$\psi_{1}(S_1 \cap U_{21}) = \psi_1(S_1) \cap \psi_2(S_2)$.
\end{enumerate}
\end{defn}
\begin{defn}\label{gcsweak}
Let $Z \subseteq X$ be a compact subset.
An {\it abstract good coordinate system of $Z$ in the weak sense}
is ${\widetriangle{\mathcal U}} = (\frak P,\{\mathcal U_{\frak p}\},\{\Phi_{\frak p\frak q}\})$
with the following properties.
\begin{enumerate}
\item
$\frak P$ is a partially ordered set.
We assume $\frak P$ is a finite set.
\item
For $\frak p \in \frak P$, $\mathcal U_{\frak p} = (U_{\frak p},S_{\frak p},\psi_{\frak p})$ is an abstract K-chart.
\item
If $\frak q \le \frak p$ then
a coordinate change $\Phi_{\frak p\frak q} = (U_{\frak p\frak q},\varphi_{\frak p\frak q})$
from $\mathcal U_{\frak q}$ to $\mathcal U_{\frak p}$ in the sense of Definition \ref{abKchange}
is defined. We require
$U_{\frak p\frak p} = U_{\frak p}$ and $\varphi_{\frak p\frak p}$ to be the
identity map.
\item
If $\frak r \le \frak q \le \frak p$ then
$
\varphi_{\frak p \frak r} = \varphi_{\frak p \frak q}\circ
\varphi_{\frak q \frak r}
$
on $U_{\frak p\frak q\frak r} := \varphi_{\frak q\frak r}^{-1}(U_{\frak p\frak q}) \cap U_{\frak p\frak r}$
(i.e., whenever both are defined).
\item
If $\psi_{\frak p}(S_{\frak p}) \cap \psi_{\frak q}(S_{\frak q})
\ne \emptyset$ then either
$\frak p \le \frak q$ or $\frak q \le \frak p$ holds.
\item
$
\bigcup \psi_{\frak p}(S_{\frak p}) \supseteq Z.
$
\end{enumerate}
\end{defn}

\begin{defn}\label{defn2424}
Let ${\widetriangle{\mathcal U}} = (\frak P,\{\mathcal U_{\frak p}\},\{\Phi_{\frak p\frak q}\})$
be an abstract good coordinate system of $Z$ in the weak sense.
We consider the disjoint union
$
\coprod_{\frak p} U_{\frak p}
$
and define a relation $\sim$ on it as follows.
Let $x \in U_{\frak p}, y \in U_{\frak q}$. We say $x\sim y$ if one of the following holds.
We put $\Phi_{\frak p\frak q} = (U_{\frak p\frak q},\varphi_{\frak p\frak q})$.
\begin{enumerate}
\item[(a)] $\frak p = \frak q$ and $x = y$.
\item[(b)] $\frak p \le \frak q$, $x \in U_{\frak q\frak p}$  and $y = \varphi_{\frak q\frak p}(x)$.
\item[(c)] $\frak q\le \frak p$, $y \in U_{\frak p\frak q}$ and $x = \varphi_{\frak p\frak q}(y)$.
\end{enumerate}
\end{defn}

\begin{defn}\label{defstrong}
An abstract good coordinate system of $Z$ in the weak sense
${\widetriangle{\mathcal U}}
= (\frak P,\{\mathcal U_{\frak p}\},\{\Phi_{\frak p\frak q}\})$
is said to be
an {\it abstract good coordinate system of $Z$ in the strong sense}
if the following holds.
\begin{enumerate}
\item[7)] The relation $\sim$ is an equivalence relation.
\item[8)]
The quotient space
$
(\coprod_{\frak p} U_{\frak p})/\sim
$
is Hausdorff with respect to the quotient topology.
\end{enumerate}
We denote by $\vert {\widetriangle{\mathcal U}} \vert$
the quotient space
$
(\coprod_{\frak p} U_{\frak p})/\sim
$
equipped with quotient topology.
\end{defn}

\begin{rem}\label{remark26}
Suppose $\frak p < \frak q < \frak r$ and
$x \in U_{\frak p}$, $y \in U_{\frak q}$, $z \in U_{\frak r}$.
We assume $x \sim y$ and $x \sim z$. Then, by definition,
$x \in U_{\frak q\frak p}$, $y = \varphi_{\frak q\frak p}(x)$.
Moreover
$x \in U_{\frak r\frak p}$, $z = \varphi_{\frak r\frak p}(x)$.
Therefore if $y \in U_{\frak r\frak q}$ in addition,
then Definition \ref{gcsweak} 4) implies
$z = \varphi_{\frak r\frak q}(y)$, and hence
$z \sim y$. Namely the transitivity holds in this case.
\par
However $y \in U_{\frak r\frak q}$ may not be satisfied in general.
This is a reason why Definition \ref{defstrong} 7)
does not follow from Definition \ref{gcsweak} 1) - 6).
\end{rem}
\begin{exm}\label{exm22}
Suppose $\frak P = \{1,2\}$ with $1<2$,
$U_1 = U_2 =\R$, $U_{21} = (-1,1)$.
$\varphi_{21} : (-1,1) \to \R$ is the {inclusion} map.
We also take $S_1 = S_2 = X = Z = \{0\}$ and
$\psi_1 = \psi_2$ is the identity map.
\par
They satisfy Definition \ref{gcsweak} 1) - 6) and
Definition \ref{defstrong} 7).
However the space $U_1 \sqcup U_2/\sim$ is not Hausdorff.
In fact $1 \in U_1$ and $1 \in U_2$ do not have separating neighborhoods.
\end{exm}

\begin{defn}\label{defn2626}
\begin{enumerate}
\item
Let $V$ be an open subset of a separable metrizable space
$U$.
We say that  $V$ is a \emph{shrinking} of $U$ and write $V \Subset U$, if $V$ is relatively compact in $U$, i.e., the closure $\overline V$ in $U$ is
compact.\footnote{We remark in a rare situation where $V$ is both
open and compact it may happen $V \Subset U$ and $V= U$.}
\item
Let $\mathcal U = (U,S,\psi)$
be an abstract K chart and $U_0 \subseteq  U$
be an open subset. We put
$\mathcal U\vert_{U_0} = (U_{0},S \cap U_{0},\psi\vert_{S \cap U_0})$.
This is an abstract K chart.
If $U_0 \Subset U$,
we say $\mathcal U\vert_{U_0}$ is a {\it shrinking} of  $\mathcal U$.

\item
Let
$\widetriangle{\UU}
= (\frak P, \{\mathcal U_\frak p\},
\{\Phi_{\frak p\frak q}\})$ be an abstract good coordinate system of $Z$ in the weak sense.
We say an abstract good coordinate system $\widetriangle{\UU}^0 = (\frak P, \{\mathcal U_\frak p^0\},
\{\Phi_{\frak p\frak q}^0\})$ of $Z$ in the weak sense is a shrinking of $\widetriangle{\UU}$ if the following hold:
\begin{enumerate}
\item
Each of $\mathcal U_\frak p^0$ is a shrinking of $\mathcal U_\frak p$
\item For $\frak p \geq \frak q$,
the domain of $\Phi_{\frak p\frak q}^0$ is a shrinking of the domain of
$\Phi_{\frak p\frak q}$ and $\Phi_{\frak p\frak q}^0$ is a restriction of
$\Phi_{\frak p\frak q}$
\end{enumerate}
\end{enumerate}
\end{defn}
\begin{thm}[Shrinking Lemma]\label{thmmain}
Suppose $\widetriangle{\UU} = (\frak P,\{\mathcal U_{\frak p}\},\{\Phi_{\frak p\frak q}\})$
is an abstract good coordinate system of $Z$ in the weak sense.
Then
there exists a shrinking $\widetriangle{\UU}^0$ of $\widetriangle{\UU}$
that becomes an abstract good coordinate system of $Z$ in the strong sense.
\end{thm}
\begin{rem}
Suppose $(V,E,\Gamma,s,\psi)$ is a Kuranishi neighborhood in the sense of
\cite[Definition A1.1]{fooobook2} or \cite[Definition 6.1]{FO}.
Then the triple
$(V/\Gamma,s^{-1}(0)/\Gamma,\psi)$ is an abstract K-chart in the sense of
Definition \ref{abKchart}.
It is easy to see that a coordinate change in the sense of
\cite[(A1.12)]{fooobook2} or \cite[Definition 6.1]{FO} induces a coordinate change
in the sense of Definition \ref{abKchange}.\footnote{
Note that Definition \ref{abKchange} 4) is required for coordinate changes
appearing in good coordinate systems.}
\par
Thus a good coordinate system in the sense of
\cite[Lemma A1.11]{fooobook2} or \cite[Definition 6.1]{FO}
induces an abstract good coordinate system
in the weak sense (of $X$) of Definition \ref{gcsweak}.
\par
The two conditions 7), 8) appearing in Definition \ref{defstrong}
is exactly the same as the conditions 7), 8) in \cite[Definition 3.14]{foootech2}.
\par
Thus Theorem \ref{thmmain} implies that we can always shrink
a good coordinate system in the sense of
\cite[Lemma A1.11]{fooobook2} or \cite[Definition 6.1]{FO}
to obtain one in the sense of \cite[Definition 3.14]{foootech2}.
\par
Note Theorem \ref{thmmain} is used  during the proof of
\cite[Theorem 3.30]{foootech2}, which claims the existence of
good coordinate system.
\par
As for a similar point on the paper \cite{foootech}, see
Remark 2.8  and Proposition 6.1 of the preprint version
 arXiv:1405.1755v1 of this paper.
\end{rem}
We will also prove the following:
\begin{prop}\label{prop51}
Let ${\widetriangle{\mathcal U}} = (\frak P,\{\mathcal U_{\frak p}\},\{\Phi_{\frak p\frak q}\})$
be an abstract good coordinate system
in the strong sense of $Z$. Let $U'_{\frak p}\Subset U_{\frak p}$
chosen for each $\frak p$. (Here  $\mathcal U_{\frak p} = (U_{\frak p},S_{\frak p},\psi_{\frak p})$.)
We consider the image $U'_{\frak p} \to \vert {\widetriangle{\mathcal U}}  \vert$
and denote it by the same symbol $U'_{\frak p}$. Then the union
$$
U' = \bigcup_{\frak p \in \frak P} U'_{\frak p} \subseteq \vert {\widetriangle{\mathcal U}}  \vert
$$
is separable and metrizable with respect to the induced topology.
\end{prop}

\section{Proof of the main theorem}\label{maintop}

\begin{lem}\label{lem3131}
Let
$\widetriangle{\UU}
= (\frak P, \{\mathcal U_\frak p\},
\{\Phi_{\frak p\frak q}\})$ be an abstract good coordinate system of $Z$ in the weak sense
and $U_{\frak p}^0 \subseteq U_{\frak p}$, $U_{\frak p\frak q}^0 \subseteq U_{\frak p\frak q}$
be open subsets.
We assume
\begin{equation}\label{assum1lem}
\varphi_{\frak p\frak q}^{-1}(U^0_{\frak p})
\cap U^0_{\frak q}
\cap S_{\frak q}
\subseteq
U^0_{\frak p\frak q}
\subseteq 
\varphi_{\frak p\frak q}^{-1}(U^0_{\frak p})
\cap U^0_{\frak q}.
\end{equation}
for $\frak q \le \frak p$ and 
\begin{equation}\label{assum2lem}
\bigcup_{\frak p \in \frak P} \psi_{\frak p}(U_{\frak p}^0) \supseteq Z.
\end{equation}
Then 
 $\widetriangle{\UU_0} = (\frak P, \{\mathcal U_\frak p\vert_{U_{\frak p}^0}\},
\{\Phi_{\frak p\frak q}\vert_{U_{\frak p\frak q}^0}\})$
is an abstract good coordinate system of $Z$ in the weak sense.
\end{lem}
\begin{proof}
We first show that 
$\Phi_{\frak p\frak q}\vert_{U_{\frak p\frak q}^0}$
is a coordinate change: $\mathcal U_\frak q\vert_{U_{\frak q}^0} \to \mathcal U_\frak p\vert_{U_{\frak p}^0}$.
Definition \ref{abKchange} 1), 2) are obvious.
Definition \ref{abKchange} 3) follows from
$$
\aligned
\varphi_{\frak p\frak q}^{-1}(U^0_{\frak p} 
\cap S_{\frak p})
\cap U^0_{\frak p\frak q}
&=
\varphi_{\frak p\frak q}^{-1}(S_{\frak p})
\cap 
\varphi_{\frak p\frak q}^{-1}(U^0_{\frak p})
\cap U^0_{\frak p\frak q} \\
&=
S_{\frak q} \cap U_{\frak p\frak q}
\cap 
\varphi_{\frak p\frak q}^{-1}(U^0_{\frak p})
\cap U^0_{\frak p\frak q} \\
&=
S_{\frak q} 
\cap 
\varphi_{\frak p\frak q}^{-1}(U^0_{\frak p})
\cap 
 U^0_{\frak p\frak q} \\
 &=
S_{\frak q} \cap U^0_{\frak q}
\cap 
\varphi_{\frak p\frak q}^{-1}(U^0_{\frak p})
\cap 
 U^0_{\frak p\frak q}
 \\
 &=
 S_{\frak q} \cap U^0_{\frak p\frak q}.
\endaligned
$$
The second equality is Definition \ref{abKchange} 3) 
for $\Phi_{\frak p\frak q}$ and the last equality follows 
from the second inclusion of (\ref{assum1lem}).
\par
We next prove  Definition \ref{abKchange} 4).
Let $\frak q\le \frak p$. (\ref{assum1lem}) implies
$$
S_{\frak q}  \cap  U^0_{\frak q}
\cap \varphi^{-1}_{\frak p\frak q}(U^0_{\frak p})
=
S_{\frak q}  \cap  U^0_{\frak p\frak q}
$$
Therefore using the fact $\varphi^{-1}_{\frak p\frak q}(S_{\frak p})
\subseteq S_{\frak q}$, we have
$$
S_{\frak q}  \cap  U^0_{\frak q}
\cap \varphi^{-1}_{\frak p\frak q}(S_{\frak p}\cap U^0_{\frak p})
=S_{\frak q}  \cap  U^0_{\frak p\frak q}.
$$
Thus Definition \ref{abKchange} 4) holds.
\par
We thus checked Definition \ref{gcsweak} 3).
Definition \ref{gcsweak} 1),2),4),5) 
follow from the corresponding properties of 
$\widetriangle{\UU}$.
Definition \ref{gcsweak} 6) is a consequence of 
(\ref{assum2lem}).
\end{proof}
\begin{lem}\label{lemKexi}
Let
$\widetriangle{\UU}
= (\frak P, \{\mathcal U_\frak p\},
\{\Phi_{\frak p\frak q}\})$ be an abstract good coordinate system of $Z$ in the weak sense.
Then there exist compact subsets $K_{\frak p} \subseteq X$ such that
\begin{equation}\label{eq33}
\bigcup_{\frak p\in\frak P} K_{\frak p} \supseteq Z, 
\qquad
K_{\frak p} \subseteq  \psi_{\frak p}(S_{\frak p}).
\end{equation}
\end{lem}
\begin{proof}
Since 
$
\bigcup_{\frak p\in\frak P} \psi_{\frak p}(S_{\frak p}) \supseteq Z
$
is an open covering, for any $x \in Z$ there exist its neighborhood
$U_x$ and $\frak p(x) \in \frak P$ such that 
$U_x \Subset \psi_{\frak p(x)}(S_{\frak p(x)})$.
We cover our compact set $Z$ by finitely many $\{U_{x_{\ell}} \mid 
\ell = 1,\dots, L\}$ of them.
Then
$
K_{\frak p} := \bigcup_{\ell; \frak p(x_\ell) = \frak p}  \overline U_{x_{\ell}}
$
has the required properties.
\end{proof}
\begin{prop}\label{prop1}
Any  abstract good coordinate system of $Z$ in the weak sense
has a shrinking.
\end{prop}
\begin{proof}
Let $\widetriangle{\UU}
= (\frak P, \{\mathcal U_\frak p\},
\{\Phi_{\frak p\frak q}\})$ be an abstract good coordinate system of $Z$ in the weak sense.
We take compact subsets $K_{\frak p}$ satisfying (\ref{eq33}).
Since $\psi_{\frak p}$ is a topological embedding $\psi_{\frak p}^{-1}(K_{\frak p})$
is compact.  There exists $U_{\frak p}^0$
such that
$\psi_{\frak p}^{-1}(K_{\frak p}) \subseteq U^0_{\frak p} \Subset U_{\frak p}$, 
since $U_{\frak p}$ is locally compact.
Then (\ref{assum2lem})
is satisfied. We put
\begin{equation}\label{defA}
A^0_{\frak p\frak q} =S_{\frak q} \cap \varphi_{\frak p\frak q}^{-1}(U^0_{\frak p}) \cap U_{\frak q}^0.
\end{equation}
Let $A_{\frak p\frak q}$ be its closure in $U_{\frak q}$.
\begin{lem}\label{lem1}
$A_{\frak p\frak q} \subseteq U_{\frak p\frak q}$ and is compact.
\end{lem}
\begin{proof}
Let $x_a \in A^0_{\frak p\frak q}$ be a sequence.
We will prove that it has a subsequence which converges to an element of  $U_{\frak p\frak q}$.
Since $x_a \in U_{\frak q}^0 \Subset U_{\frak q}$
we may assume that $x \in U_{\frak q}$ is its limit.
By definition of $ A^0_{\frak p\frak q}$, $y_a:= \varphi_{\frak p\frak q}(x_a) \in S_{\frak p} \cap U^0_{\frak p}$.
Since $ U^0_{\frak p}$ is relatively compact in $U_{\frak p}$, there is a subsequence of $\{y_a \}$ such that
it converges to some $y \in U_{\frak p}$.
On the other hand, by  Definition \ref{abKchange} 3), $\psi_{\frak p}(y_a) = \psi_{\frak q}(x_a)$.
Then by continuity of $\psi_{\frak p}: S_{\frak p} \to X$, $\psi_{\frak q}: S_{\frak q} \to X$,
$\psi_{\frak q}(x) = \psi_{\frak p}(y)$. (We use the fact that $X$ is Hausdorff here.) 
Obviously this point is contained in
$\psi_{\frak p}(S_{\frak p}) \cap \psi_{\frak q}(S_{\frak q})$ which is equal to $\psi_{\frak q}(S_{\frak q} \cap U_{\frak p\frak q})$ by  Definition \ref{abKchange}
4). By the injectivity of $\psi_{\frak q}$ on $S_{\frak q}$, this implies
$x \in U_{{\frak p\frak q}}$. This finishes the proof.
\end{proof}

 Using Lemma \ref{lem1} and the local 
 compactness of $U_{\frak p\frak q}$, we then take $V^0_{\frak p\frak q}$ such that
\begin{equation}\label{U121}
A_{\frak p\frak q} \subseteq V^0_{\frak p\frak q}  \Subset U_{\frak p\frak q}
\end{equation}
and put 
$$
U_{\frak p\frak q}^0 = V^0_{\frak p\frak q} \cap 
\varphi_{\frak p\frak q}^{-1}(U^0_{\frak p})
\cap U^0_{\frak q}.
$$
Since $A_{\frak p\frak q}^0 \subseteq \varphi_{\frak p\frak q}^{-1}(U^0_{\frak p})
\cap U^0_{\frak q}$,
(\ref{U121}) implies
$
A^0_{\frak p\frak q} \subseteq U^0_{\frak p\frak q}  \Subset U_{\frak p\frak q}.
$
Since $U_{\frak p}^0$ and $U_{\frak p\frak q}^0$ satisfy (\ref{assum1lem}) 
and (\ref{assum2lem}), Proposition \ref{prop1} follows from Lemma \ref{lem3131}.
\end{proof}

We start the proof of the main theorem.
We take a shrinking 
 $\widetriangle{\UU_1} = (\frak P, \{\mathcal U_\frak p\vert_{U_{\frak p}^1}\},
\{\Phi_{\frak p\frak q}\vert_{U_{\frak p\frak q}^1}\})$
of given
$\widetriangle{\UU}
= (\frak P, \{\mathcal U_\frak p\},
\{\Phi_{\frak p\frak q}\})$.
We put
\begin{equation}
\varphi^1_{\frak p\frak q} = \varphi_{\frak p\frak q}\vert_{U_{\frak p\frak q}^1}.
\end{equation}
We apply Lemma \ref{lemKexi} to  $\widetriangle{\UU_1}$ to obtain $K_{\frak p}.$ 
We take a metric $d_{\frak p}$ of $U_{\frak p}$
and put:

\begin{equation}\label{defnnnn}
U_{\frak p}^{\delta} = \{ x \in U^1_{\frak p} \mid d_{\frak p}(x,\psi_{\frak p}^{-1}(K_{\frak p})) < \delta\}.
\end{equation}
Since $\psi_{\frak p}^{-1}(K_{\frak p})$ is compact
and $U^1_{\frak p}$ is locally compact,
$U_{\frak p}^{\delta} 
\Subset U^1_{\frak p}$ for sufficiently small $\delta$.
\par
We use the next lemma several times in this section.
\begin{lem}\label{convinientlem}
Suppose $\frak q \le \frak p$, 
$\delta_n \to 0$ and $x_n \in {U_{\frak q}^{\delta_n}}
\cap (\varphi^1_{\frak p\frak q})^{-1}({U_{\frak q}^{\delta_n}})$.
Then there exists a subsequence of $x_n$, still denoted by $x_{n}$,
such that:
\begin{enumerate}
\item $x_{n}$ converges to $x \in S_{\frak q}$.
\item
$\varphi_{\frak p\frak q}^1(x_{n})$ converges to $y \in S_{\frak p}$.
\item
$\psi_{\frak q}(x) = \psi_{\frak p}(y) \in K_{\frak p}\cap K_{\frak q}$.
\item
$x \in U^1_{\frak p\frak q}$ and $y =\varphi^1_{\frak p\frak q}(x)$.
\end{enumerate}
\end{lem}
\begin{proof}
Let $\delta_0 > 0$ be a fixed sufficiently small constant such that
$U_\frak p^{\delta_0} \subseteq U_\frak p^1$, and consider $\delta> 0$
with $\delta < \delta_0$.
Since $U^{\delta}_{\frak p} \Subset U^{\delta_0}_{\frak p}$
and $U^{\delta}_{\frak q} \Subset U^{\delta_0}_{\frak q}$
for small $\delta$, we may take a subsequence such that 
$x_{n}$ and $\varphi_{\frak p\frak q}^1(x_{n})$ converge to $x
\in U^{\delta_0}_{\frak q}$ and $y \in U^{\delta_0}_{\frak p}$, respectively.
\par
Then (\ref{defnnnn}) implies $x \in \psi_{\frak q}^{-1}(K_{\frak q})$
and $y \in \psi_{\frak p}^{-1}(K_{\frak p})$.
We have proved 1), 2).
\par
Since $x_n \in U^1_{\frak p\frak q} \Subset U_{\frak p\frak q}$,
its limit $x$ is in  $U_{\frak p\frak q}$.
Since $\varphi_{\frak p\frak q}$ is defined on $U_{\frak p\frak q}$
and is continuous, we have
$
\varphi_{\frak p\frak q}(x) = \varphi_{\frak p\frak q}(\lim_{n\to \infty} x_n)
= \lim_{n\to \infty}\varphi_{\frak p\frak q}(x_n)
= y.
$
Then by Definition \ref{abKchange} 3) 
we have $\psi_{\frak q}(x) = \psi_{\frak p}(y)$.
Note $\psi_{\frak q}(x) \in K_{\frak q}$ 
and  $\psi_{\frak q}(y) \in K_{\frak q}$.
Therefore 3) holds.
\par
Then $x \in U^1_{\frak p\frak q}$ follows from 
Definition \ref{abKchange} 4)
and $K_\frak p \subseteq \psi_\frak p(S_\frak p \cap U_\frak p^1)$, 
$K_\frak q \subseteq \psi_\frak q(S_\frak q \cap U_\frak q^1)$.
\end{proof}
We take a decreasing sequence of positive numbers $\delta_n$ 
with $\lim_{n\to\infty}\delta_n = 0$ and put
\begin{equation}\label{formula39}
U^n_{\frak p} = U_{\frak p}^{\delta_n}, \qquad 
U^n_{\frak p\frak q} = U_{\frak q}^{\delta_n} \cap (\varphi^{1}_{\frak p\frak q})^{-1}(U_{\frak p}^{\delta_n}).
\end{equation}
We remark 
$
U^n_{\frak p\frak q} \subseteq U^1_{\frak p\frak q}
$
since $U^1_{\frak p\frak q}$ is the domain of $\varphi^{1}_{\frak p\frak q}$.
\par
By  Lemma \ref{lem3131}, $\widetriangle{\UU_n} = (\frak P, \{\mathcal U_\frak p\vert_{U_{\frak p}^n}\},
\{\Phi_{\frak p\frak q}\vert_{U_{\frak p\frak q}^n}\})$
is an abstract good coordinate system of $Z$ in the weak sense.
Since $U_{\frak p}^{n} \subseteq U_{\frak p}^{1} \Subset U_{\frak p}$
and $U_{\frak p\frak q}^{n} \subseteq U_{\frak p\frak q}^{1} \Subset U_{\frak p\frak q}$,
$\widetriangle{\UU_n}$ is a shrinking of $\widetriangle{\UU}$.
\par
We will prove that $\widetriangle{\UU_n}$ 
is an abstract good coordinate system of $Z$ in the {\it strong} sense
for sufficiently large $n$.
The proof occupies the rest of this section.
We put
\begin{equation}\label{313}
C^n_{\frak p} = \overline{U_{\frak p}^{n}}, \qquad 
C^n_{\frak p\frak q} = \overline{U_{\frak q}^{n}} \cap (\varphi^{1}_{\frak p\frak q})^{-1}(\overline{U_{\frak p}^{n}}).
\end{equation}
Here $\overline{U_{\frak p}^{n}}$ is the closure of $U_{\frak p}^{n}$ in $U_{\frak p}$, 
which coincides with the closure of  $U_{\frak p}^{n}$ in $U_{\frak p}^1$.
(This is because $U^n_{\frak p} \Subset U^1_{\frak p}$.)
Moreover $C^n_{\frak p}$ is compact.
We consider
$$
\hat U^n = \coprod_{\frak p\in \frak P}  U_{\frak p}^{n},
\qquad
\hat C^n =\coprod_{\frak p\in \frak P}  C_{\frak p}^{n}
$$
where the right hand sides are disjoint union.
Note $\hat U^n \subseteq \hat C^n$.
We define a relation on $\hat U$ by applying Definition \ref{defn2424} to $\widetriangle{\UU_n}$ .
We denote it by $\sim_n$.
We also define a relation $\sim'_n$ on $\hat C^n$ as follows. 
\begin{defn}
Let $x \in C^n_{\frak p}$ and $y \in C^n_{\frak q}$.
 We say $x\sim'_n y$ if one of the following holds. \begin{enumerate}
\item[(a)] $\frak p = \frak q$ and $x = y$.
\item[(b)] $\frak p \le \frak q$, 
$x \in C^n_{\frak q\frak p}$ and $y = \varphi^1_{\frak q\frak p}(x)$.
\item[(c)] $\frak q\le \frak p$, 
$y \in C^n_{\frak p\frak q}$ and $x = \varphi^1_{\frak p\frak q}(y)$.
\end{enumerate}
\end{defn}
The next lemma is immediate from our choice (\ref{formula39}) of $U_{\frak p\frak q}^n$.
\begin{lem}\label{onaji}
Let $x,y \in \hat U^n \subseteq \hat C^n$. Then 
$x \sim_n y$ if and only if $x \sim'_n y$.
\end{lem}
We now prove:
\begin{prop}
The relations
$\sim_n$ and $\sim_n'$ are equivalence relations for sufficiently large $n$.
\end{prop}
\begin{proof}
In view of Lemma \ref{onaji} it suffices to show that 
$\sim_n'$ is an equivalence relation for sufficiently large $n$.
\par
We assume that this is not the case.
Note $\sim'_n$ satisfies all the property required for equivalence relation 
possibly except transitivity.
Therefore by taking a subsequence if necessary 
we may assume that there exist $x_n,y_n,z_n \in \hat C^n$ such that 
$x_n \sim_n' y_n$, $y_n \sim_n' z_n$ but $x_n \sim'_n z_n$ does not hold.
\par
Let $x_n \in C^n_{\frak p_n}$, $y_n \in C^n_{\frak q_n}$, $z_n \in C^n_{\frak r_n}$.
Since $\frak P$ is a finite set we may assume, by taking a subsequence if necessary, 
that $\frak p = \frak p_n$, $\frak q = \frak q_n$, $\frak r = \frak r_n$ are independent of $n$.
\par
We remark that $C^n_{\frak p} \subseteq U^{2\delta_n}_{\frak p}$.
Therefore 
we apply Lemma \ref{convinientlem} to $x_n$ and can take a subsequence 
such that $\lim_{n\to \infty} x_n = x$ 
and $y = \lim_{n\to \infty}y_n$  with 
$x \in U_{\frak p\frak q}^1$ and
$
\psi_{\frak p}(x) = \psi_{\frak q}(y).
$
\par
We can again apply Lemma \ref{convinientlem} with $x_n$, $\frak p$, $\frak q$ 
replaced by $y_n$, $\frak q$, $\frak r$, respectively.
Then by taking a subsequence if necessary we have 
$z = \lim_{n\to \infty}z_n$, such that 
$y \in U_{\frak r\frak p}^1$
and 
$
\psi_{\frak q}(y) = \psi_{\frak r}(z).
$
\par
Thus we have $\psi_{\frak p}(x) = \psi_{\frak q}(y) = \psi_{\frak r}(z)$.
Therefore either $\frak p \le \frak r$ or $\frak r \le \frak p$ holds.
We may assume $\frak r\le \frak p$ without loss of generality.
Then since $\psi_{\frak p}(x) = \psi_{\frak r}(z)$
we have $z \in U^1_{\frak p\frak r}$, 
$\varphi_{\frak p\frak r}(z) = x$ by 
Definition \ref{abKchange} 3), 4).
Therefore 
$z_n \in U_{\frak p\frak r}^1$
for sufficiently large $n$, since $U_{\frak p\frak r}^1$ is open in $U_\frak r$.
We use it to show:
\begin{lem}\label{lem31010}
We have $\varphi^1_{\frak p\frak r}(z_n) = x_n$
for sufficiently large $n$.
\end{lem}
\begin{proof}
Since $\psi_{\frak p}(x) = \psi_{\frak q}(y) = \psi_{\frak r}(z)$
Definition \ref{gcsweak} 5) and $\frak r \le \frak p$ imply that 
one of the following holds.
\par (a)
$\frak q \le \frak r \le \frak p$.
(b)
$\frak r \le \frak q \le \frak p$.
(c)
$\frak r \le \frak p \le \frak q$.
\par
In Case (a) we have
$y \in U^1_{\frak r \frak q} \cap
U^1_{\frak p \frak q} \cap
(\varphi^1_{\frak r \frak q})^{-1}(U^1_{\frak p \frak r}).$
Hence  for all sufficiently large $n$,
$y_n \in U^1_{\frak r \frak q} \cap
U^1_{\frak p \frak q} \cap
(\varphi^1_{\frak r \frak q})^{-1}(U^1_{\frak p \frak r})$
and
$x_n = \varphi^1_{\frak p \frak q}(y_n)
=\varphi^1_{\frak p \frak r} \circ \varphi^1_{\frak r \frak q}(y_n)
=\varphi^1_{\frak p \frak r}(z_n)$, by 
Definition \ref{gcsweak} 
4).\par
In Case (b), we have
$z \in U^1_{\frak p \frak r} \cap U^1_{\frak q \frak r}
\cap (\varphi^1_{\frak q \frak r})^{-1}(U^1_{\frak p \frak q})$.
Hence, for all sufficiently large $n$,
$z_n \in U^1_{\frak p \frak r} \cap U^1_{\frak q \frak r}
\cap (\varphi^1_{\frak q \frak r})^{-1}(U^1_{\frak p \frak q})$
and
$\varphi^1_{\frak p \frak r}(z_n)
=\varphi^1_{\frak p \frak q} \circ \varphi^1_{\frak q \frak r}(z_n)
=\varphi^1_{\frak p \frak q}(y_n)
=x_n$.
\par
In Case (c) we have
$z \in U^1_{\frak p \frak r} \cap U^1_{\frak q \frak r}
\cap (\varphi^1_{\frak p \frak r})^{-1}(U^1_{\frak q \frak p}).$
Hence, for sufficiently large $n$,
$z_n \in U^1_{\frak p \frak r} \cap U^1_{\frak q \frak r}
\cap (\varphi^1_{\frak p \frak r})^{-1}(U^1_{\frak q \frak p})$.
Moreover 
$y_n = \varphi_{\frak q\frak p}^1(x_n)$ and
$y_n= \varphi^1_{\frak q \frak r}(z_n)=\varphi^1_{\frak q \frak p} \circ
\varphi^1_{\frak p \frak r}(z_n)$.
Since $\varphi^1_{\frak q \frak p}$ is injective,
we find that
$x_n = \varphi^1_{\frak p \frak r}(z_n)$.
\end{proof}
Lemma \ref{lem31010} implies $x_n \sim_n' z_n$ 
for sufficiently large $n$. This is a contradiction.
\end{proof}
We thus have proved that $\widetriangle{\mathcal U_n}$ satisfies  Definition \ref{defstrong} 7) 
for sufficiently large $n$.
We turn to the proof of
 Definition \ref{defstrong} 8).
Let  $W_{\frak p\frak q} \Subset U_{\frak p\frak q}^1$ be an open neighborhood of 
$\psi_{\frak q}^{-1}(K_{\frak p} \cap K_{\frak q})$.
\begin{lem}\label{lem35}
For sufficiently small $\delta$ we have
\begin{equation}\label{form3911}
(\varphi^1_{\frak p\frak q})^{-1}(U_{\frak p}^{\delta}) \cap U_{\frak q}^{\delta} \subseteq 
W_{\frak p\frak q}.
\end{equation}
\end{lem}
\begin{proof}
If (\ref{form3911}) is false there exists $\delta_n > 0$ and
$
x_n \in \left((\varphi^1_{\frak p\frak q})^{-1}(U_{\frak p}^{\delta_n}) \cap U_{\frak q}^{\delta_n}\right) \setminus W_{\frak p\frak q}
$
with $\delta_n \to 0$. 
We apply Lemma \ref{convinientlem} and 
may assume  1), 2), 3), 4) of Lemma \ref{convinientlem}.
Then
$x \in U^1_{\frak p\frak q}$
$,
\psi_{\frak q}(x) = \psi_{\frak p}(y) \in K_{\frak q}\cap K_{\frak p}.
$
It implies  $x \in W_{\frak p\frak q}$.
Thus $x_n \in W_{\frak p\frak q}$ for large $n$.
This is a contradiction.
\end{proof}

\begin{lem}\label{lem310}
$C^n_{\frak p\frak q}$ is a compact subset of $C^n_{\frak q}$ 
for sufficiently large $n$.
\end{lem}
\begin{proof}
It suffices to show that $C^n_{\frak p\frak q}$ is a closed subset of $C^n_{\frak q}$.
Let $x_a \in C^n_{\frak p\frak q}$ be a sequence converging to $x \in C^n_{\frak q}$.
By definition
\begin{equation}\label{form3939}
x_ a \in  \overline{U^{\delta_n}_{\frak q}} \cap (\varphi^1_{\frak p \frak q})^{-1}( \overline{U^{\delta_n}_{\frak p}}).
\end{equation}
Now (\ref{form3939}), (\ref{form3911}) and $\overline{U^{\delta_n}_{\frak q}} \subseteq {U^{2\delta_n}_{\frak q}}$  imply that 
$
x_a \in W_{\frak p\frak q} \Subset U^1_{\frak p\frak q}
$
for sufficiently large $n$.
Therefore $x \in U^1_{\frak p\frak q}$.
Since $\varphi^1_{\frak p\frak q}$ is continuous on $U^1_{\frak p\frak q}$ we have
$
\lim_{a\to\infty}\varphi^1_{\frak p\frak q}(x_a) = \varphi^1_{\frak p\frak q}(x).
$
Since $\varphi^1_{\frak p\frak q}(x_a) \in C^n_{\frak p}$ and $C^n_{\frak p}$ is compact,
 $x \in U^1_{\frak p\frak q}$ implies   $\varphi^1_{\frak p\frak q}(x) \in C^n_{\frak p}$.
Thus $x \in C_{\frak p\frak q}^n$. This proves that $C_{\frak p\frak q}^n$ is closed in $C_\frak q^n$ as required.
\end{proof}
We define
$
C^n = \hat C^n/\sim_n' 
$.
\begin{lem}\label{hasdorn}
The space $C^n$ is Hausdorff with respect to the quotient topology.
\end{lem}
This is a standard consequence of Lemma \ref{lem310}.
We remark that $\vert\widetriangle{\mathcal U^n}\vert = \hat U^n/\sim_n$
by definition.
The inclusion $\hat U^n \to \hat C^n$ induces a map 
$\hat U^n \to C^n$. 
Lemma \ref{onaji} implies that it induces an {\it injective} map
$\vert\widetriangle{\mathcal U^n}\vert \to C^n$.
This map is continuous by the definition of the quotient topology.
Therefore Lemma \ref{hasdorn} implies that 
$\vert\widetriangle{\mathcal U^n}\vert$ is Hausdorff.
The proof of Theorem \ref{thmmain} is now complete.

\begin{rem} We would like to note that the domain 
$U_{\frak p\frak q}^n$ of the coordinate change 
of the shrinking 
$\widetriangle{\mathcal U_n}$ of $\widetriangle{\mathcal U}$
is {\it not} of the form
\begin{equation}\label{tanjyaaunin}
\varphi_{\frak p\frak q}^{-1}(U_{\frak p}^n) 
\cap U_{\frak q}^n
\end{equation}
but is
$$
U_{\frak p\frak q}^n 
= (\varphi_{\frak p\frak q}^1)^{-1}(U_{\frak p}^n) 
\cap U_{\frak q}^n
=
\varphi_{\frak p\frak q}^{-1}(U_{\frak p}^n) 
\cap U_{\frak q}^n \cap U_{\frak p\frak q}^1.
$$
In fact (\ref{tanjyaaunin}) 
is {\it not}
relatively compact in $U_{\frak p\frak q}$ in general. We thank J. Solomon who found an example
to clarify this point and informed it to us.
\end{rem}

\section{Proof of metrizability}
\label{section5}
In this section we prove Proposition \ref{prop51}.

We recall the following well-known definition.
A family of subsets $\{U_i \mid i \in I\}$ of a topological space $Y$ containing $x \in Y$
is said to be a neighborhood basis of $x$ if
\begin{enumerate}
\item[(nbb 1)] each $U_i$ contains an open neighborhood of $x$,
\item[(nbb 2)] for each open set $U$ containing $x$ there exists $i$ such that $U_i \subseteq U$.
\end{enumerate}
A family of open subsets $\{U_i \mid i \in I\}$ of a topological space $X$ is said to be a
basis of the open sets if for each $x$ the set $\{U_i \mid x \in U_i\}$
is a neighborhood basis of $x$. A topological space is said to satisfy the second axiom of countability
if there exists a countable basis of open subsets $\{U_i \mid i \in I\}$.
A classical result of Urysohn says a topological space is metrizable
if it is regular and satisfies the second axiom of countability.
(See a standard text book such as \cite{kelly} for these facts.)
\begin{proof}[Proof of Proposition \ref{prop51}]
We put $K_{\frak p} = \overline{U'_{\frak p}}$
and consider
$K = \coprod_{\frak p \in \frak P} K_{\frak p}/\sim_K$ in $\vert {\widetriangle{\mathcal U}} \vert$.
(Here $\sim_K$ is the restriction of the equivalence relation 
$\sim_U$ obtained by applying Definition \ref{defn2424} to 
$\widetriangle{\mathcal U}$. 
($\sim_U$ is an equivalence relation on 
$\coprod_{\frak p \in \frak P} U_{\frak p} \supseteq \coprod_{\frak p \in \frak P} K_{\frak p}$.)
Let
$
\Pi_{\frak p} :K_{\frak p} \to K
$
be the %\red{projection} 
%\blue
{the natural inclusion followed by the projection. 
As a subset of $\vert {\widetriangle{\mathcal U}} \vert$, we can also write
$K = \bigcup_{\frak p \in \frak P} K_{\frak p} \subseteq \vert {\widetriangle{\mathcal U}} \vert$.}
Note the induced topology of the embedding $U' \to K$
coincides with the induced topology of the embedding
$U' \to \vert {\widetriangle{\mathcal U}}  \vert$.
This is because the map $K \to \vert {\widetriangle{\mathcal U}}  \vert$
is a topological embedding. ($K$ is compact and $\vert {\widetriangle{\mathcal U}} \vert$  is
Hausdorff.)
Therefore, it suffices to show that $K$ is metrizable with respect to the quotient topology
of
$
\Pi_{\frak P,K}:
\coprod_{\frak p \in \frak P} K_{\frak p} \to K.
$
We remark that $K$ is compact.
$K$ is Hausdorff since $\vert {\widetriangle{\mathcal U}}  \vert$ is Hausdorff and
$K \to \vert {\widetriangle{\mathcal U}}  \vert$ is injective and continuous.
Therefore $K$ is regular.
Now it remains to show that $K$ satisfies the second axiom of countability.
%\blue{Some reference on general topology would be useful for those who forget %about the
%definitions of these general topology concepts. We rarely use these in our daily 
%life.}
This is \cite[Lemma 8.5]{foootech}. We repeat its proof here for the convenience
of the reader.
\par
For each $\frak p$, we take a countable basis $\frak U_{\frak p}
=\{U_{\frak p,i_{\frak p}} \subseteq K_{\frak p} \mid i_{\frak p}  \in I_{\frak p}\}$
of open sets of $K_{\frak p}$. We may assume $\emptyset \in \frak U_{\frak p}$.
\par
For each $\vec i = (i_{\frak p})_{\frak p \in \frak P}$ ($i_{\frak p} \in I_{\frak p}$) we define
$U(\vec i)$ to be the interior of the set
\begin{equation}\label{U+basis1}
U^+(\vec i) := \bigcup_{\frak p \in \frak P}\Pi_{\frak p}(U_{\frak p,i_{\frak p}}).
\end{equation}
Then $\{ U({\vec i}) \}$  is a countable family of open subsets of $K$.
We will prove that this family is a basis of open sets of $K$.
\par
Let $q \in K$, we put
\begin{equation}\label{defPxsss}
\frak P(q)
=
\{\frak p \in \frak P \mid \exists x, \,\,  q=[x], \, x \in K_{\frak p}\}.
\end{equation}
Here and hereafter we identify $K_{\frak p}$ to the image of $\Pi_{\frak P,K}(K_{\frak p})$ in $K$.
Note since $K$ is Hausdorff and $K_{\frak p}$ is compact,
the natural inclusion map $K_{\frak p} \to \coprod_{\frak p \in\frak P} K_{\frak p}$
induces a topological embedding $K_{\frak p} \to K$.
\par
For $\frak p \in \frak P(q)$, we have $x_{\frak p} \in K_{\frak p}$ with $[x_\frak p] = q$.
We put
$$
I_{\frak p}(q) = \{i_{\frak p} \in I_{\frak p} \mid x_{\frak p} \in U_{\frak p,i_{\frak p}}\}.
$$
Then $\{U_{\frak p,i_{\frak p}} \mid i_{\frak p} \in I_{\frak p}(q)\}$ is a countable neighborhood basis of $x_{\frak p}$
in $K_{\frak p}$.
For each $\vec i = (i_{\frak p}) \in \prod_{\frak p\in \frak P(q)} I_{\frak p}(q)$,
we set
\begin{equation}\label{U+basis12}
U^+(\vec i) =\bigcup_{\frak p \in \frak P(q)} \Pi_{\frak p}(U_{\frak p,i_{\frak p}})
\subseteq K.
\end{equation}
We claim that the collection $\{U^+(\vec i)  \mid \vec i \in \prod_{\frak p\in \frak P(q)} I_{\frak p}(q)\}$
is a neighborhood basis of $q$ in $K$ for any $q$.
The claim follows from Lemmata \ref{9sublem1},\, \ref{9sublem2} below.
\par
\begin{lem}\label{9sublem1} The subset
$U^+(\vec i)$ is a neighborhood of $q$ in $K$.
\end{lem}
\begin{proof}
For $\frak p \in \frak P(q)$ the set $K_{\frak p} \setminus U_{\frak p,i_{\frak p}}$
is a closed subset of $K_{\frak p}$ and so is compact.
Therefore $\Pi_{\frak p}(K_{\frak p} \setminus U_{\frak p,i_{\frak p}})$ is a compact subset in the Hausdorff space $K$ and so is closed.
\par
If $\frak p \notin \frak P(q)$ then we consider  $\Pi_{\frak p}(K_{\frak p})$ which is closed.
\par
Now we put
$$
K_0 = \bigcup_{\frak p \in \frak P(q)}\Pi_{\frak p}(K_{\frak p} \setminus U_{\frak p,i_{\frak p}})
\cup
\bigcup_{\frak p \notin \frak P(q)}\Pi_{\frak p}(K_{\frak p}).
$$
This is a finite union of closed sets and so is closed.
It is easy to see that
$
q \in K \setminus K_0 \subseteq U^+(\vec i).
$
\end{proof}
\begin{lem} \label{9sublem2} The collection $\{U^+(\vec i)\}$ satisfies
the property (nbb 2)  of the neighborhood basis above.
\end{lem}
\begin{proof}
Let $U \subseteq K$ be an open subset containing $q$.
{Since}  the map $K_{\frak p} \to K$ is a topological embedding,
$U\cap K_{\frak p}$ is an open set of $K_{\frak p}$.
Therefore for each $\frak p \in \frak P(q)$,
the set  $U\cap K_{\frak p}$ is a neighborhood of $x_{\frak p}$
in $K_{\frak p}$. By the definition of neighborhood basis in $K_\frak p$,
there exists $i_{\frak p}$ such that $U_{\frak p,i_{\frak p}} \subseteq U\cap K_{\frak p}$.
We put $\vec i = (i_{\frak p})$. Then
$U^+(\vec i) \subseteq U$ as required.
\end{proof}
We remark that $U^+(\vec i)$ in (\ref{U+basis12}) is a special case of $U^+(\vec i)$ in
(\ref{U+basis1}).
(We take $U_{\frak p,i_{\frak p}} = \emptyset$ for $\frak p \notin \frak P(x)$.)
The family $U(\vec i)$ is a countable basis of open sets of $K$.
Proposition \ref{prop51} is {now} proved.
\end{proof}
\begin{rem}\label{rem5252}
Note $U'$  can also be written as $\coprod U'_{\frak p}/\sim$ for certain
equivalence relation $\sim$. However we do not equip it with the quotient topology
but with the subspace topology of the quotient topology on $\vert {\widetriangle{\mathcal U}} \vert$.
\end{rem}
\vspace{0.1in}
\noindent
{\bf Acknowledgement.} \ \  The authors would like to thank Dominic Joyce and Jake Solomon
for helpful comments.
Kenji Fukaya is supported partially by JSPS Grant-in-Aid for Scientific Research
No. 23224002 and NSF Grant No. 1406423,  Yong-Geun Oh by the IBS project IBS-R003-D1,
Hiroshi Ohta by JSPS Grant-in-Aid
for Scientific Research No. 23340015,  and Kaoru Ono by JSPS Grant-in-Aid for
Scientific Research, Nos. 26247006, 23224001.
\bibliographystyle{amsalpha}

\end{document}